\documentclass[reqno]{amsart}
\usepackage{amsfonts}
\usepackage{amsmath}
\usepackage{amsthm}
\usepackage{amssymb}
\usepackage{amscd}
\usepackage{mathtools}
\usepackage{mathrsfs}
\usepackage[latin1,utf8]{inputenc}
\usepackage{fancyhdr}
\usepackage[hang,flushmargin]{footmisc} 

\usepackage{graphicx}
\usepackage{float}
\usepackage[dvipsnames]{xcolor}
\usepackage{pgf,tikz,pgfplots}
\pgfplotsset{compat=1.13}
\usetikzlibrary{arrows}

\usepackage[numbers]{natbib}

\usepackage{hyperref}
\hypersetup{
    colorlinks=true,
    linkcolor=Violet,
    citecolor=Violet,
    urlcolor=Violet,
}

\definecolor{ccqqww}{rgb}{0.8,0.,0.4}
\definecolor{ffccww}{rgb}{1.,0.8,0.4}
\definecolor{qqwwzz}{rgb}{0.,0.4,0.6}
\definecolor{ttzzqq}{rgb}{0.2,0.7,0.4}

\theoremstyle{plain}
\newtheorem{theorem}{Theorem}[section]
\newtheorem{lemma}[theorem]{Lemma}
\newtheorem{proposition}[theorem]{Proposition}
\newtheorem{proposition*}{Proposition}
\newtheorem{corollary}[theorem]{Corollary}

\theoremstyle{definition}
\newtheorem*{definition}{Definition}
\newtheorem{example}{Example}

\theoremstyle{remark}
\newtheorem{remark}{Remark}

\numberwithin{equation}{section}

\DeclareMathOperator{\R}{\mathbb{R}}
\DeclareMathOperator{\C}{\mathbb{C}}
\DeclareMathOperator{\complexset}{\mathbb{C}}
\DeclareMathOperator{\quatset}{\mathbb{H}}

\DeclareMathOperator{\K}{\mathbb{K}}

\DeclareMathOperator{\A}{\mathcal{A}}

\DeclareMathOperator{\Lw}{\mathtt{L}_\mathnormal{w}}

\makeatletter
\newcommand{\ltext}[2]{%
  \@bsphack
  \csname phantomsection\endcsname 
  \def\@currentlabel{#1}{\label{#2}}%
  \@esphack
}
\makeatother

\title[Generic spectrum of the weighted Laplacian on finite groups]{Generic spectrum of the weighted Laplacian operator on Cayley graphs}

\author[C. F. Coletti \and
L. R. de Lima \and
D. S. de Oliveira \and
M. A. M. Marrocos]
{Cristian F. Coletti \and
Lucas R. de Lima \and
Diego S. de Oliveira \and
Marcus A. M. Marrocos}

\address{Centro de Matem\'atica, Computa\c{c}\~ao e Cogni\c{c}\~ao, Universidade Federal do ABC\\
Av. dos Estados, 5001\\
09210-580 Santo Andr\'e, S\~ao Paulo\\
Brazil.}
\email{cristian.coletti@ufabc.edu.br}
\email{lucas.roberto@ufabc.edu.br}
\email{diego.sousa@ufabc.edu.br}

\address{Departamento de Matem\'atica, Instituto de Ci\^encias Exatas, Universidade Federal do Amazonas\\
Av. Rodrigo Otávio, 6200\\
69080-900 Manaus, Amazonas\\
Brazil.}
\email{marcusmarrocos@ufam.edu.br}
\thanks{{\bf Funding:} This study was financed in part by the Coordenação de Aperfeiçoamento de Pessoal de Nível Superior - Brasil (CAPES) - Finance Code 001. It was also supported by grants \#2017/10555-0 and \#2019/19056-2 S\~ao Paulo Research Foundation (FAPESP)}



\begin{document}

\begin{abstract}
In this paper, we investigate the spectrum of a class of weighted Laplacians on Cayley graphs and determine under what conditions the corresponding eigenspaces are generically irreducible. Specifically, we analyze the spectrum on left-invariant Cayley graphs endowed with an invariant metric, and we give some criteria for generically irreducible eigenspaces. Additionally, we introduce an operator that is comparable to the Laplacian and show that the same criterion holds.
\end{abstract} 

\maketitle

\section{Introduction}

Spectra of graphs have received much attention since its foundations were laid in the fifties and sixties of the 20\textsuperscript{th} century. The interest in spectral properties can be explained by the large number of related applied problems. In this work, the considered spectrum of a graph is the set of eigenvalues of a ``Laplacian" operator. However, there exists no canonical notion of Laplacian operator on graphs, although some of them are quite common. The Laplacian operator that we assume here is a weighted Laplacian $\Delta_w$ on graphs acting on real or complex functions defined on its vertices. The operator $\Delta_w$ applied to a function $f$ can be written as
\begin{equation*}
    \Delta_w f(x) =  -\sum\limits_{y \sim x} w(x,y)\big( f(y) - f(x) \big).
\end{equation*}
Here, $x$ and $y$ are vertices of the graph, $y \sim x$ stands for the adjacency of $y$ and $x$, \textit{i.e.}, the existence of the edge $\{x,y\}$, and $w(x,y)$ is a positive weight of $\{x,y\}$.

The basic spectral properties for this operator on finite graphs are well known in the literature. 

A remarkable and well-known result on generic properties of the Laplacian operator on Riemannian manifolds proved by Uhlenbeck \cite{uhlenbeck1976} shows that the eigenvalues are generically simple in the space of $ C^k $-Riemannian metrics. Roughly speaking, it means that up to a small perturbation of the metric the eigenvalues are always simple. Recently, many works have addressed this question on the setting of graphs, see \cite{berkolaiko2017,berkolaiko2018,friedlander2005,poignard2017}. 

Now suppose that we consider only metrics containing symmetries given by a group $G$. Then each corresponding Laplace operator commutes with such symmetries. This fact implies that every composition of an eigenfunction by a given symmetry is again an eigenfunction with the same eigenvalue. Thus, the presence of symmetries may increase the multiplicities of the eigenvalues. In this context the real eigenspaces of the Laplacian are orthogonal representations of the group $G$ and the multiplicities of the eigenvalues are at least the dimensions of such representations. This situation occurs, for instance, when in a Riemannian Manifold one restricts to metrics such that the isometry group contains a fixed group $G$. So we can not expect that the eigenvalues are generically simple in this space of metrics. However, we can ask if the spectrum is generically $G$-simple, that is, if all eigenspaces are generically irreducible representations of $G$.
More specifically, a (real or complex) linear operator $T: \mathcal{U} \to \mathcal{U}$ or its spectrum is called \textit{(real or complex)} \textit{$G$-simple} if its eigenspaces are irreducible (real or complex) representations of $G$. The question about the generic $G$-simplicity of the spectrum of a given Laplace operator has been studied in the literature in different contexts, see for instance \cite{berkolaiko2017,berkolaiko2018,marrocos2019,petrecca2019,zelditch1990}.

Zelditch \cite{zelditch1990} established the generic situation of spectrum of the Laplacian in the $G$-manifolds with a finite group $G$, under the hypothesis that dimension of $M$ is bigger than the degree of all orthogonal irreducible representation. The higher symmetry case occurs when the group $G$ acts transitively on the manifold, see \cite{schueth2017,petrecca2019}. Schueth~\cite{schueth2017} provided a completely algebraic criterion for the existence of left invariant metrics compact Lie group $G$ such that each eigenspace of the Laplacian operator has an irreducible representation under the action of $G$. In other words, the eigenspaces are $G$-simple up to small perturbations on the set of left-invariant metrics when the criterion holds, \emph{i.e.}, the eigenvalues do not exhibit higher multiplicities than the ones prescribed by the symmetries.

The generic irreducibility question  for operators on graphs has not received much atention. In our knowledge, up to now, Berkolaiko and Liu \cite{berkolaiko2018} is the only work in the setting of graphs with symmetries.  They construct a family of combinatorial Schrodinger operators on symmetric graphs on which generic irreducibility fails.  More precisely, they found a family of $G$-invariant combinatorial Laplacians on graphs where the irreducibility of the eigenspace does not occur when $G$ is the tetrahedron symmetry group. In this paper our goal is to establish the generic situation of the spectrum of the Laplacian on finite Cayley graphs having the set of weights as parameter space.  We point out that we provide families of graphs where the conjecture holds and does not hold. Our strategy consists in adapting Schueth's  method to Cayley graphs.

The strategy is to describe the Laplacian as an algebraic operator that depends only on the representations of the group. It is then possible to apply some algebraic techniques to study the multiplicities of the induced operator in each irreducible representation. 

Chung and Sternberg \cite{chung1992} considered a similar approach. They applied representation theory for Laplacians on homogeneous graphs associated with certain quadratic forms calculating their eigenvalues explicitly. They studied some general operators such as the Laplacian on vector bundles. Those operators are local and commute with the given symmetries. In their applications and examples, the authors associate the spectrum with vibrational states and modes on molecules. The distinct eigenvalues were then associated with the irreducible representations of its group of symmetries. The number of distinct eigenvalues also gives the maximum possible number of vibrational states and modes. On the other hand, one can ask if there exists an operator performing the maximum number of vibrational states, which is related to the minimal multiplicities of the eigenvalues. Here we study a subclass of homogeneous graphs, the Cayley graphs, and consider left-invariant structures by establishing a generic setting.

Let $(W,\rho)$ an arbitrary irreducible representation of a finite group $G$, \textit{i.e.}, a homomorphism $\rho:G\to\operatorname{GL}(W)$ that does not contain non-trivial subrepresentations. Set $S$ to be a symmetric generating set of $G$. In other words, $S$ generates $G$ and for all $s\in S$, $s^{-1} \in S$. Regard $S':= \big\{ \{s,s^{-1}\} : s \in S\big\}$. The left invariant weight space of a Cayley graph $\mathcal{C}(G,S)$ is 
\begin{equation} \label{eq:L_S}
    \mathcal{L}_S:=\{w = (w_{s'})_{{s'}\in S'}\in \mathbb{R}^{S'}: \forall s' \in S'(w_{s'}>0)\}\subseteq \mathbb{R}^{S'}.
\end{equation}
To simplify notation, we write $\alpha_s= \alpha_{s'}$ for all $s \in S$ and $\alpha \in \mathbb{R}^{S'}$ with $s \in s'$. Let us denote by $\mathcal{W}_S \subseteq \mathcal{L}_S$ the set of all $w \in \mathcal{L}_S$ such that $\Delta_w$ is real $G$-simple.

For each $\alpha \in \mathbb{R}^{S'}$, we define a linear map $D_W(\alpha): W \to W$ given by
\[
D_W(\alpha) =  - \sum\limits_{s \in S} \alpha_s \big(\rho(s) - \rho(e) \big)
\]
and denote its characteristic polynomial by $P_W(\alpha)$. We describe how the weighted Laplacian $\Delta_w$ associated with an arbitrary weight $w \in \mathcal{L}_S$ can be completely characterized by the operators $D_W(w)$. So, to prove that the spectrum of the weighted Laplacian is generically $G$-simple on $\mathcal{C}(G,S)$, the strategy is to minimize the multiplicities of the eigenvalues associated with the operators $D_W(w)$.

In order to achieve this goal, we use a map
$\text{res}: \complexset[x] \times \complexset[x] \longrightarrow \complexset$   which stands for the \textit{resultant} and it satisfies the following properties:
\begin{enumerate}
	\item[(i)] Given a pair of polynomials $p = \sum\limits_j a_jx^j $ and $q = \sum\limits_k b_k x^k $ then $\text{res}(p,q) = \sum\limits_{j,l,k,m} c_{j,l,k,m} a_j^l b_k^m $ for a finite number of coefficients $c_{j,l,k,m}$.
	\item[(ii)] $\text{res}(p,q) \neq 0 \Leftrightarrow p $ and $q$ do not share any common zero.
\end{enumerate} 
The existence of the resultant is a well-known fact (see \cite{gelfand1994} for more details).
Let $W, W_1$ and $W_2$ irreducible representations of $G$ and $\alpha \in \mathbb{R}^{S'}$. Denote by $P_W'(\alpha)$ and by $P_W''(\alpha)$ the first and second derivatives of the polynomial $P_W(\alpha)$, respectively. Consider the functions
\[
a_{W_1,W_2},b_{W}, c_W: \mathbb{R}^{S'} \longrightarrow \complexset
\]
given by
\begin{enumerate}
	\item[(i)] $a_{W_1,W_2}(\alpha) = \text{res}(P_{W_1}(\alpha), P_{W_2}(\alpha))$
	\item[(ii)] $b_{W}(\alpha) = \text{res}(P_{W}(\alpha), P_{W}'(\alpha))$
	\item[(ii)] $c_{W}(\alpha) = \text{res}(P_W(\alpha), P_W''(\alpha))$
\end{enumerate}
In a few words, $a_{W_1,W_2}$ detects if two distinct representations $W_1$ and $W_2$ produce the same eigenvalue or not. The maps $b_W$ and $c_W$ indicates if $W$ produces eigenvalues with multiplicity $1$ and $2$, respectively, or not.

To study the representations of $G$ is the same as to study the set of equivalence classes of complex irreducible representations of $G$ which we denote by $\operatorname{Irr}(G,\complexset)$. The set $\operatorname{Irr}(G,\complexset)$ can be partitioned into three subsets. They are $\operatorname{Irr}(G,\complexset)_{\R}$ (representations of real type), $\operatorname{Irr}(G,\complexset)_{\quatset}$ (representations of quaternionic type) and $\operatorname{Irr}(G,\complexset)_{\complexset}$ (representations of complex type). The idea is that representations of real and quaternionic type are related to representations with the set of scalars taken as $\R$ and $\quatset$ respectively, up to a process of extension or restriction on these scalars. The representations that are not of real or quaternionic types are said to be of complex type. It is due to this process that the eigenvalues provided by the representations of quaternionic type will always have at least multiplicity $2$. Another aspect regarding the representations of complex type is that it may produce the same eigenvalues as its dual. Fortunately, those properties do not produce obstacles when we consider only real-valued eigenfunctions. For more details on representations of these types see Section \ref{main results}.

We can now state one of the main results of this paper, the irreducibility criterion for Cayley graphs.

\begin{theorem} \label{thm:Cayley!generic}
	A finite Cayley graph $\mathcal{C}(G,S)$ admits a weight $w \in \mathcal{L}_S$ such that $\Delta_w$ is real $G$-simple if, and only if, the following conditions are simultaneously satisfied
	\begin{enumerate}
		\item[(i)] $a_{W_1,W_2} \not\equiv 0$ for all $W_1,W_2 \in \operatorname{Irr}(G,\complexset)$ with $W_1, W_1^* \neq W_2$.
		
		\item[(ii)] $b_W \not\equiv 0$ for all $W \in \operatorname{Irr}(G,\complexset)_{\R, \C}$.
		
		\item[(iii)]$c_W \not\equiv 0$ for all $W \in \operatorname{Irr}(G,\complexset)_{\quatset}$. 
	\end{enumerate}
	
	Moreover, the existence implies that $\mathcal{W}_S$, the set of all invariant weights $w \in \mathcal{L}_S$ that turns $\Delta_w$ into a real $G$-simple operator, is a residual set in $\mathcal{L}_S$.
\end{theorem}

The main advantage of substituting our original problem by the verification of items (i), (ii) and (iii) is that they are much more flexible.  Basically, we only need to verify that the maps $a_{W_1,W_2}$,  $b_W$ and $c_W$ do not vanish identically, when we consider the entire set $\operatorname{Irr}(G,\C)$. Since $a_{W_1,W_2}$, $b_W$ and $c_W$ are defined on $\mathbb{R}^{S'}$, it is much easier to find suitable elements $\alpha \in \mathbb{R}^{S'}$ such that (i), (ii) and (iii) hold instead of finding elements $w \in \mathcal{L}_S$ satisfying the same items. That is, it is much more advantageous to have at our disposal the entire family of operators $D_W(\alpha)$ to test the criterion instead of the family of Laplacians. Operators like $D_W(\alpha)$ with $\alpha \in \mathbb{R}^{S'}$ can be easier to describe. However, sometimes it can be hard to study the eigenvalues of the operators like $\Delta_w$. Examples presented in Section \ref{main results} illustrate how we may conveniently choose the operators of the form $D_W(\alpha)$, which are not necessarily Laplacians. We also introduce a new operator $\Lw$ motivated by the Laplacian on Lie groups for which we verify similar results in Section \ref{sec.Lw}. We apply the criteria obtained to several examples.

This paper is organized as follows. Section \ref{preliminaries} introduces the fundamental definitions and properties of weighted Laplacians on Cayley graphs. In Section 3, we prove Theorem \ref{thm:Cayley!generic}, the Schueth's criterion for Cayley graphs, and we provide several examples to demonstrate its application. Lastly, in Section 4, we define a new operator, denoted as $\Lw$, that is similar to the weighted Laplacian. We then confirm that an equivalent criterion holds for the generic spectrum of $\Lw$.

\section{Preliminaries}\label{preliminaries}
In this section, we define the essential elements to be studied.  We begin by setting up the notation and terminology for some elements of graph theory. 

Here, a \textit{graph} $\mathcal{G}=(V,E)$ consists of a non-empty set of vertices $V$ and a set of edges $E$ whose elements are $\{x,y\}$ with $x,y\in V$. We denote by $x\sim y$ the existence of the edge $\{x,y\}\in E$. A \textit{graph automorphism} is a bijection $F: V \to V$ such that, for all $\{x, y\} \in E$, one has $F(y) \sim F(x)$. The \textit{weight} of an edge $\{x,y\} \in E$ is given by a function $w: E \to (0,+\infty)$. For simplicity of notation we write $w(x,y)$.

We define below Cayley graphs, the central discrete structure studied in this article. They determine a class of graphs with particular properties of our interest in the study of the Laplacian operators, as disclosed in the next sections.

\begin{definition}[Cayley graph]
    Let $(G,.)$ be a group and let $S \subseteq G$. The left-invariant Cayley graph $\mathcal{C}(G,S)= (V,E)$ is defined by
    \[
        V=G \quad \text{and} \quad E= \big\{ \{x,xs\}: x \in G, s \in S\big\}.
    \]
\end{definition}

Throughout this paper, we regard $S$ as a symmetric generating set of $G$ such that $e \not\in S$, where $e$ is the identity element of the finite group $G$. Recall the definitions of $S'$, $\mathcal{L}_S \subseteq \mathbb{R}^{S'}$ and $\alpha_s$ given in the introduction. 

The left invariant weights along the edges are characterized by $S$ fixing $w_s = w_{s^{-1}}>0$ for each $s \in S$ and setting $w(x,xs)=w_s$ for all $x \in G$. We can now regard $w\in \mathbb{R}^{S'}$ and the set of all left invariant weights $w$ associated with $\mathcal{C}(G,S)$ is $\mathcal{L}_S \subseteq \mathbb{R}^{S'}$. Let $\mathcal{W}_S \subseteq \mathcal{L}_S$ be the set of all $\alpha \in \mathcal{L}_S$ such that $\Delta_\alpha$ is real $G$-simple.

Set $\A_{\K}$ to be the $\K$-algebra of the functions $f:V\to\K$ where $\K$ is the field $\R$ or $\C$ of real or complex numbers. When there is no need to specify $\K$, we simply write $\mathcal{A}$. Observe that every $f \in \A$ can be written as $f=\sum_{x \in V}f(x)\delta_x$ with $\delta_x(y)=\delta^x_y$ where $\delta$ stands for the Kronecker delta.

Let us define the inner product $\langle\cdot,\cdot\rangle$ on $\mathcal{A}$ by $\langle f,g \rangle:= \sum_{x \in G}f(x)\overline{g(x)}$, for all $f,g \in \mathcal{A}$. We verify that $\Delta_w$ is symmetric with respect to $\langle\cdot,\cdot\rangle$.

\begin{lemma} \label{lm:Laplacian.symmetric}
    Let $G$ be a finite group with $S$ a symmetric generating set. Then, for all $w \in \mathcal{L}_S'$ and all $f,g \in \mathcal{A}$,
    \[\langle \Delta_w f, g\rangle = \langle f,\Delta_w g\rangle.\]
\end{lemma}
\begin{proof}
    Recall that $\Delta_w f =  -\sum_{x \in G}\sum_{s \in S} w_s\big( f(xs) - f(x) \big)\delta_x$, $s^{\pm1} \in S$, and $w_s=w_{s^{-1}}$. Hence, since $G$ acts transitively on itself,
    \begin{align*}
        \langle\Delta_w f,g\rangle &= -\sum_{x \in G} \sum_{s \in S} w_s\big( f(xs) - f(x) \big)\overline{g(x)} \\
        &= -\sum_{y \in G} \sum_{s \in S} w_{s^{-1}} f(y) \overline{g(ys)} + \sum_{x \in G} \sum_{s \in S}w_{s^{-1}}f(x)\overline{g(x)} \\
        &= \sum_{x \in G} f(x)\left(\overline{-\sum_{s \in S} w_{s}\left(g(xs) -g(x)\right)}\right) = \langle f, \Delta_w g\rangle. 
    \end{align*}
\end{proof}

\section{\texorpdfstring{Criterion for $G$-simplicity on  Cayley graphs}{Criterion for {\it G}-simplicity on  Cayley graphs}}\label{main results}

We will restrict our attention to study the generic irreducibility property of the eigenspaces of a class of Laplacian operators $\Delta_w$ on finite Cayley graphs. Our goal is to establish an analogous criterion to the one obtained by Schueth \cite{schueth2017}, namely Theorem \ref{thm:Cayley!generic}. We will follow the notation used in \cite{schueth2017}, adapting it to our purposes if necessary. 

Throughout this section, consider $\mathcal{C}(G,S)$ a finite and connected Cayley graph with symmetric subset of generators $S \subset G$ such that $e \notin S$. Take $\mathcal{A}$ as the complex function space $L^2(G,\C)$, $\mathcal{A}_{\R}$ as the real function space $L^2(G,\R)$,  and $\mathcal{L}_S$ the left-invariant weight space as in equation \ref{eq:L_S}.

A representation of the group $G$ on a vector space $W$ is a group homomorphism $\rho: G \to \operatorname{GL}(W,\K)$. Let us define for each $y \in G$ the function $r_y \in \mathcal{A}$ to be such that $r_y(x)= xy$ for all $x \in G$. The right regular representation $R$ on $\A$ is given by $R(y)(f) = f \circ r_y$. Therefore, one can easily see that
\begin{equation} \label{laplacian-R}
    \Delta_w f = -\sum\limits_{x\in G; \; s \in S} w_s \big(f \circ r_s - f \big)(x)\delta_x.
\end{equation}

The left regular representation $L$ on $\A$ is given by $L(y)(f) = f \circ \ell_{y^{-1}}$ where $\ell_{y^{-1}}(x)=y^{-1}x$. Let $\mathcal{U}$ a subspace of $\A$. A $\K$-linear operator $T: \mathcal{U} \to \mathcal{U}$ is called \textit{$G$-simple} if its eigenspaces are irreducible subrepresentations of $(\A,L)$. If $\K = \R$ we call it \textit{real $G$-simple}. 

\begin{remark}
Our goal is to establish conditions in order to determine when $\Delta_w$ is a real $G$-simple operator on the given Cayley graph. 
\end{remark}

$\operatorname{Irr}(G,\C)$ will denote the set of nonequivalent irreducible representations of $G$.

 Let $(W,\rho) \in \operatorname{Irr}(G,\C)$. We know from representation theory of compact groups that there exists a $G$-isomorphism $\varphi_W: W^* \otimes W \to \operatorname{Im}(\varphi_W) \subseteq \A$ such that, for all $x \in G$, one has
\begin{equation} \label{isomorphism S}
id \otimes \rho(x) = \varphi_W^{-1} \circ R(x) \circ  \varphi_W
\end{equation}
(see for instance \cite[Ch.~III]{broecker1995} and \cite{schueth2017}).

$\operatorname{Im}(\varphi_W)$ is denoted by $\operatorname{I}(W)$ and it is called the \textit{$W$-isotypical component} of the right regular representation, which coincides with the $W^*$-isotypical component of the left regular representation. This a consequence of the Peter-Weyl theorem applied to the right and left regular representations.

Now, we define the types of complex representations.

\begin{definition}
Let $(W, \rho)$ be a irreducible complex representation of $G$. We say that $W$ is of \textit{real type} (respectively of \textit{quaternionic type}) if it admits a conjugate linear $G$-map $J: W \longrightarrow W$ such that $J^2 = id$ (respectively $J^2 = -id$). If $W$ is not of one of those types, it is called of \textit{complex type}.
\end{definition}

We mention that for each $W \in \operatorname{Irr}(G, \C)$ there exists an irreducible real $G$-representation $W_{\R}$ such that 
\begin{equation} \label{eq:real-form-representations}
    \C \otimes W_{\R} \simeq \left\{ \begin{array}{l}
     \text{$W$, if $W$ is of real type,} \\
     \text{$W \oplus W^*$, if $W$ is of complex or quaternionic type}
\end{array} \right. 
\end{equation}
(see \cite[Sec.2]{schueth2017} or \cite[Ch.II, Sec. 6]{broecker1995}).

We will denote by $\operatorname{Irr}(G, \C)_{\mathbb{D}}$ the irreducible representations of type $\mathbb{D}$.

Let $W \in \operatorname{Irr}(G,\complexset)$. The subrepresentation $C_W$ is given by
\[
C_W := \left\{ \begin{array}{ccc}
\operatorname{I}(W) & \leftrightarrow & W\in \operatorname{Irr}(G;\mathbb{C})_{\mathbb{R}} \\
\operatorname{I}(W) & \leftrightarrow & W\in \operatorname{Irr}(G;\mathbb{C})_{\mathbb{H}} \\
\operatorname{I}(W)\oplus \operatorname{I}(W^*) & \leftrightarrow & W\in \operatorname{Irr}(G;\mathbb{C})_{\mathbb{C} \, .}
\end{array} \right.
\]
We also define $\mathcal{E}_W:= C_W \cap \A_{\R}$.

\begin{proposition} \label{W to isotypical}
For the Cayley graph $\mathcal{C}(G,S)$, let $(W,\rho) \in \operatorname{Irr}(G, \C)$ be a complex irreducible subrepresentation of $(\A,R)$. Then, for a given $w \in \mathcal{L}_S$, we have $\Delta_w(W)\subset W$ and 
\[
id \otimes \Delta_w^W = \varphi_W^{-1} \circ \Delta_w|_{\operatorname{I}(W)} \circ  \varphi_W,
\]
where the linear operator $\Delta_w^W: W \to W$ is given by
\begin{equation*} 
    \Delta_w^W := {-} \sum\limits_{s\in S} w_s \big(\rho(s) - id \big) \,.
\end{equation*}
As an immediate consequence, we conclude that, up to isomorphisms, the spectrum of the restricted operator $\Delta_w|_{I(W)}$ coincides with the spectrum of
\[
id \otimes \Delta_w^W:  W^* \otimes W \to W^* \otimes W \, .
\]
Moreover, up to multiplicities, the eigenvalues of $\Delta_w|_{I(W)}$ and $\Delta_g^W$ are the same.
\end{proposition}

\begin{proof}
The result desired follows directly from \eqref{laplacian-R} and \eqref{isomorphism S}.
\end{proof}

We remember the Peter-Weyl decomposition for the complex function space $\A = L^2(G, \C)$, that is,
\[
\A \simeq \bigoplus_{(W,\rho) \in \operatorname{Irr}(G,\C)} W^* \otimes W \, .
\]
Proposition \ref{W to isotypical} basically states that we can identify $\Delta_w$ by using this decomposition as the sum operator
\[
\bigoplus_{(W,\rho)\in \operatorname{Irr}(G,\C)} id \otimes \Delta_w^W: \bigoplus_{(W,\rho) \in \operatorname{Irr}(G,\C)} W^* \otimes W \to \bigoplus_{(W,\rho) \in \operatorname{Irr}(G,\C)} W^* \otimes W \, .
\]

As a consequence, we present below three corollaries, each pertaining to a distinct type of representation. It is worth noting that if $\lambda$ is an eigenvalue of $\Delta_w|_{C_W}$, then it is also an eigenvalue of $\Delta_w^W$.

\begin{corollary} \label{corollary:eigenspace-real-type} For the Cayley graph $\mathcal{C}(G,S)$, let $(W,\rho) \in \operatorname{Irr}(G, \C)_{\R}$ be an irreducible subrepresentation of real type in $(\A,R)$ and $w \in \mathcal{L}_S$. Let $\lambda$ be an eigenvalue of $\Delta_w|_{C_W}$. Consider $(C_W)_{\lambda}$ and $(\mathcal{E}_W)_{\lambda}$ to be the eigenspaces of $\Delta_w|_{C_W}$ and $\Delta_w|_{\mathcal{E}_W}$, respectively. Suppose that $m(\lambda)$ is the multiplicity of $\lambda$ for the operator $\Delta_w^W$. Then, 
\[
(C_W)_{\lambda} \simeq W^* \otimes \C^{m(\lambda)} \simeq W^{\oplus m(\lambda)}\, . 
\]
In particular, $\Delta_w|_{\mathcal{E}_W}$ has a real $G$-simple spectrum if, and only if, $\Delta_w^W$ has simple spectrum. 
\end{corollary}

\begin{proof}
For $W$ of real type, we have $C_W = I(W)$ and $W \simeq W^*$. Since $m(\lambda)$ is the multiplicity of the $\lambda$-eigenvalue of $\Delta_w^W$, then its corresponding eigenspace is isomorphic to $\C^{m(\lambda)}$. By Proposition \ref{W to isotypical}, the $\lambda$-eigenspace of $\Delta_w |_{C_W}$ coincides (up to isomorphisms) with the $\lambda$-eigenspace of 
\[
id \otimes \Delta_w^W:  W^* \otimes W \to W^* \otimes W \, ,
\]
which is given by $W^* \otimes \C^{m(\lambda)}$. Since $W^* \otimes \C^{m(\lambda)}$ is isomorphic to the $G$-module $ (W^*)^{\oplus m(\lambda)}$, and $W \simeq W^*$ for $W$ of real type, we conclude the first part of the proof.

We will now proceed to prove the second part of the corollary. For each eigenvalue $\lambda$ of $\Delta_w|_{C_W}$,  fix a real irreducible representation $W_{\R}$ such that $\C \otimes W_{\R} \simeq W$, as in equation \ref{eq:real-form-representations}. Since the complex $G$-module $W$ is the simple complexification of the real $G$-module $W_{\R}$ and $\Delta_w|_{\mathcal{E}_W}$ is the real version of $\Delta_w|_{C_W}$, then the $\lambda$-eigenspace of $\Delta_w|_{\mathcal{E}_W}$ is isomorphic to  $W_{\R}^{\oplus m(\lambda)}$. In particular, $\Delta_w|_{\mathcal{E}_W}$ has a real $G$-simple spectrum if, and only if, $m(\lambda) = 1$ for each eigenvalue $\lambda$, which is the same to say that $\Delta_w^W$ has simple spectrum. 
\end{proof}

\begin{corollary} \label{corollary:eigenspace-complex-type} For the Cayley graph $\mathcal{C}(G,S)$, let $(W,\rho) \in \operatorname{Irr}(G, \C)_{\C}$ be an irreducible subrepresentation of complex type in $(\A,R)$ and $w \in \mathcal{L}_S$. Let $\lambda$ as an eigenvalue of $\Delta_w|_{C_W}$. Consider $(C_W)_{\lambda}$ and $(\mathcal{E}_W)_{\lambda}$ to be the eigenspaces of $\Delta_w|_{C_W}$ and $\Delta_w|_{\mathcal{E}_W}$, respectively. Suppose that $m(\lambda)$ is the multiplicity of $\lambda$ for the operator $\Delta_w^W$. Then, 
\[
(C_W)_{\lambda} \simeq (W \oplus W^*) \otimes \C^{m(\lambda)} \simeq (W \oplus W^*)^{\oplus m(\lambda)}\, . 
\]
In particular, $\Delta_w|_{\mathcal{E}_W}$ has a real $G$-simple spectrum if, and only if, $\Delta_w^W$ has simple spectrum. 
\end{corollary}

\begin{proof}
For $W$ of complex type, we have $C_W = I(W)\oplus I(W^*)$. Since $m(\lambda)$ is the multiplicity of the $\lambda$-eigenvalue of $\Delta_w^W$, then its corresponding eigenspace is isomorphic to $\C^{m(\lambda)}$. We will now prove that
\begin{equation} \label{eq:comutatividade-laplaceW-laplaceW*}
\eta \circ \Delta_w^W = \Delta_w^{W^*} \eta, \quad \forall \eta \in W^* 
\end{equation}
Recall that, for each $s \in S$, we have $w_s = w_{s^{-1}}$. Thus, for every $v \in W$,
\[
\begin{array}{ccl}
\eta \circ \Delta_w^W(v)  &=& \eta\big( \Delta_w^W v \big)  \\
     &=&  - \eta \left( \sum\limits_{s\in S} w_s \big(\rho(s)v - v \big) \right) \\
     &=&  -  \sum\limits_{s\in S} w_s\big( \, \eta(\rho(s)v) - \eta(v)  \, \big) \\
     &=&  - \sum\limits_{s\in S} w_s\big( \, \eta(\rho(s)v) - \eta(v)  \, \big) \\
     &=&  - \sum\limits_{s\in S} w_s\big( \, (\rho^*(s^{-1})\eta) (v) - \eta(v)  \, \big) \\
     &=&  - \sum\limits_{s\in S} w_{s^{-1}}\big( \, \rho^*(s^{-1}) - id  \, \big) (\eta)(v) \\
     &=& (\Delta_w^{W^*}\eta)(v) 
\end{array}
\]

Equation \ref{eq:comutatividade-laplaceW-laplaceW*} implies that $\Delta_w^W$ and $\Delta_w^{W^*}$ share the same eigenvalues with their respective multiplicities. Therefore, the $\lambda$-eigenspace of $\Delta_w^{W^*}$ is also isomorphic to $\C^{m(\lambda)}$.

By Proposition \ref{W to isotypical}, the $\lambda$-eigenspace of $\Delta_w |_{C_W}$ coincides (up to isomorphisms) with the $\lambda$-eigenspace of 
\[
(id \otimes \Delta_w^W) \oplus (id \otimes \Delta_w^{W^*}):  (W^* \otimes W)\oplus (W \otimes W^*) \to (W^* \otimes W)\oplus (W \otimes W^*)\, ,
\]
which is given by $(W^* \otimes \C^{m(\lambda)}) \oplus (W \otimes \C^{m(\lambda)}) $. Thus, the $\lambda$-eigenspace of $\Delta_w |_{C_W}$ coincides with $(W \oplus W^*) \otimes \C^{m(\lambda)} \simeq (W \oplus W^*)^{\oplus m(\lambda)}$, thereby concluding the initial segment of the proof.

Let us prove the second part of the statement. For each eigenvalue $\lambda$ of $\Delta_w|_{C_W}$,  fix a real irreducible representation $W_{\R}$ such that $\C \otimes W_{\R} \simeq (W \oplus W^*)$, as in equation \ref{eq:real-form-representations}. Since the sum of complex $G$-modules $(W \oplus W^*)$ is the complexification of a single  real $G$-module $W_{\R}$ and $\Delta_w|_{\mathcal{E}_W}$ is the real version of $\Delta_w|_{C_W}$, then the $\lambda$-eigenspace of $\Delta_w|_{\mathcal{E}_W}$ is isomorphic to  $W_{\R}^{\oplus m(\lambda)}$. In particular, $\Delta_w|_{\mathcal{E}_W}$ has a real $G$-simple spectrum if, and only if, $m(\lambda) = 1$ for each eigenvalue $\lambda$, which is the same to say that $\Delta_w^W$ has simple spectrum. 
\end{proof}

\begin{corollary} \label{corollary:eigenspace-quaternionic-type} For the Cayley graph $\mathcal{C}(G,S)$, let $(W,\rho) \in \operatorname{Irr}(G, \C)_{\mathbb{H}}$ be an irreducible subrepresentation of quaternionic type in $(\A,R)$ and $w \in \mathcal{L}_S$. Let $\lambda$ be an eigenvalue of $\Delta_w|_{C_W}$. Consider $(C_W)_{\lambda}$ and $(\mathcal{E}_W)_{\lambda}$ to be the eigenspaces of $\Delta_w|_{C_W}$ and $\Delta_w|_{\mathcal{E}_W}$, respectively. Suppose that $m(\lambda)$ is the multiplicity of $\lambda$ for the operator $\Delta_w^W$. Then, $m(\lambda)$ is even and
\[
(C_W)_{\lambda} \simeq (W \oplus W) \otimes \C^{\, m(\lambda)/2} \simeq (W \oplus W)^{\oplus \, m(\lambda)/2}\, . 
\]
In particular, $\Delta_w|_{\mathcal{E}_W}$ has a real $G$-simple spectrum if, and only if, each eigenvalue of $\Delta_w^W$ has multiplicity $2$. 
\end{corollary}

\begin{proof}
For $W$ of quaternionic type, we have $C_W = I(W)$. Since $m(\lambda)$ is the multiplicity of the $\lambda$-eigenvalue of $\Delta_w^W$, then its corresponding eigenspace is isomorphic to $\C^{m(\lambda)}$. Considere the operator $\Delta_w^{W\oplus W}$ given by
\[
\Delta_w^{W\oplus W} := - \sum\limits_{s \in S} w_s \big( \, (\rho \oplus \rho)(s) - id \, \big): W\oplus W \to W\oplus W 
\]
It is easy to show that $\Delta_w^{W\oplus W}$ coincides with the product operator $\big( \Delta_w^W(\cdot)\, , \, \Delta_w^W(\cdot) \big)$.

It is well-known from the properties of representations of quaternionic type that we can choose the structural quaternionic map $j:W \oplus W \to W \oplus W$ as
\[
j(u,v) := (-v,u)\, , \quad \forall u,v \in W 
\]
(see more in \cite[Ch.2, Sec.6]{broecker1995} ). Note that if $u$ is a $\lambda$-eigenvector of $\Delta_w^W$, then 
\[
\begin{array}{ccl}
    \Delta_w^{W \oplus W} \, j(u,0) &=& \Delta_w^{W \oplus W} (0,u) \\
     &=& (0, \Delta_w^W u) \\
     &=& (0, \lambda u) \\
     &=& \lambda (0, u) \\
     &=& \lambda \, j(u,0)\,. \\
\end{array}
\]
In particular, $(u,0)$ is an $\lambda$-eigenvector if, and only if, $j(u,0)$ is an $\lambda$-eigenvector. Similarly, $(0,v)$ is an $\lambda$-eigenvector if, and only if, $j(0,v)$ is an $\lambda$-eigenvector. This $j$-invariance of the $\lambda$-eigenspace implies that $m(\lambda)$ is even. Therefore, we can write $(W\oplus W)^{\oplus \, m(\lambda)/2} = W^{\oplus m(\lambda)}$.

By Proposition \ref{W to isotypical}, the $\lambda$-eigenspace of $\Delta_w |_{C_W}$ coincides (up to isomorphisms) with the $\lambda$-eigenspace of 
\[
id \otimes \Delta_w^W:  W^* \otimes W \to W^* \otimes W\, ,
\]
which is given by $W^* \otimes \C^{m(\lambda)} $. Since $W\simeq W^*$ for $W$ of quaternionic type, then $W^* \otimes \C^{m(\lambda)} \simeq (W^*)^{\oplus m(\lambda)}$ is isomorphic to the $G$-module $ (W \oplus W)^{\oplus \, m(\lambda)/2}$, which concludes the first part of the demonstration.

We now turn to the proof of the second part of the corollary. For each eigenvalue $\lambda$ of $\Delta_w|_{C_W}$,  fix a real irreducible representation $W_{\R}$ such that $\C \otimes W_{\R} \simeq (W \oplus W^*)$, as in equation \ref{eq:real-form-representations}. Thus, $\C \otimes W_{\R} \simeq (W \oplus W)$. Since the sum of complex $G$-modules $(W \oplus W)$ is the complexification of a single  real $G$-module $W_{\R}$ and $\Delta_w|_{\mathcal{E}_W}$ is the real version of $\Delta_w|_{C_W}$, then the $\lambda$-eigenspace of $\Delta_w|_{\mathcal{E}_W}$ is isomorphic to  $W_{\R}^{\oplus \, m(\lambda)/2}$. In particular, the operator $\Delta_w|_{\mathcal{E}_W}$ has a real $G$-simple spectrum if, and only if, $m(\lambda) = 2$ for each eigenvalue $\lambda$. 
\end{proof}

The three precedent corollaries imply:

\begin{corollary}
	For the Cayley graph $\mathcal{C}(G,S)$, let $w \in \mathcal{L}_S$ be a left-invariant weight. Then $\Delta_w$ is real $G$-simple iff the following statements are simultaneously satisfied
	\begin{enumerate}
		\item[(i)] For all $W_1,W_2 \in \operatorname{Irr}(G, \complexset)$ with $W_1,W_1^* \ncong W_2$,  $\Delta_w^{W_1}$ and $\Delta_w^{W_2}$ do not share any common eigenvalue.
		
		\item[(ii)] For each $W \in \operatorname{Irr}(G, \complexset)_{\R, \C}$, all eigenvalues of $\Delta_w^W$ have multiplicity one.
		
		\item[(iii)] For each $W \in \operatorname{Irr}(G, \complexset)_{\quatset}$, all eigenvalues of $\Delta_w^W$ have multiplicity two.
	\end{enumerate}
	\label{corolario 3.1 - Schueth}
	
\end{corollary}

\subsection{Proof of the main criterion for finite Cayley graphs}

Let $(W, \rho)$ a representation, $\alpha= (\alpha_{s'})_{s'\in S'} \in \mathbb{R}^{S'}$  and $D_W(\alpha): W \longrightarrow W$ given by
\begin{equation} \label{DW operator}
D_W(\alpha) =  - \sum\limits_{s \in S} \alpha_s \big(\rho(s) - id \big)
\end{equation}

Equation \eqref{DW operator} induces a map $D_W: \mathbb{R}^{S'} \to \operatorname{End}(W)$ for each $(W, \rho) \in \operatorname{Irr}(G, \complexset)$.
In particular, from Proposition \ref{W to isotypical} one has
\[
D_W\big( w \big) = \Delta_w^W.
\]

Thus, every left-invariant weighted Laplacian operator can be studied as a $D_W(\alpha)$ for some $\alpha \in \mathcal{L}_S$. Let
\[
P_W : \mathbb{R}^{S'} \longrightarrow \complexset[x]
\]
be the application that maps $\alpha \in \mathcal{L}_S$ to the characteristic polynomial of $D_W(\alpha)$.

We can now prove the main criterion for the generic simplicity. Recall the definitions of $a_{W_1,W_2}$, $b_W$, and $c_W$ given in the introduction.

\begin{proof}[Proof of Theorem \ref{thm:Cayley!generic}]
If $w \in \mathcal{L}_S$ is such that $\Delta_w$ is real $G$-simple, then $(i)$, $(ii)$, and $(iii)$ follow immediately from Corollary \ref{corolario 3.1 - Schueth}. Now, assume that $(i)$, $(ii)$, and $(iii)$ are jointly satisfied. Note that the elements $\alpha \in \mathbb{R}^{S'} $ for that this hypothesis does not hold are exactly the inverse image at zero of any of $a_{W_1,W_2}$, or $b_W$, or $c_W$ from items $(i)$, $(ii)$, and $(iii)$. But since $\operatorname{Irr}(G, \C)$ is finite for the finite group $G$ (see \cite[Thm.~7]{serre1977}), then the inverse image of these polynomials at $0$ yields a meager finite set $\mathcal{N}$ in $\mathbb{R}^{S'}$. We also have that $\mathcal{N} \cap \mathcal{L}_S$ a meager set in $\mathcal{L}_S$. Thus the set $\mathcal{L}_S\hspace{-4pt}\setminus \mathcal{N}$ is residual in $\mathcal{L}_S$. Moreover, Corollary \ref{corolario 3.1 - Schueth} also ensures that $\mathcal{L}_S\hspace{-4pt}\setminus \mathcal{N}$ contains exactly the weights $\alpha$ that turns $\Delta_\alpha$ into a real $G$-simple operator, \textit{i.e.}, $\mathcal{W}_S = \mathcal{L}_S\hspace{-4pt}\setminus \mathcal{N}$.

\end{proof}

\begin{remark} \label{rem:Ssubset:irr}
Note that $\Delta_w$ depends on the choice of $S \subseteq G\setminus\{e\}$, and so does $a_{W_1,W,2}$, $b_{W}$ and $c_W$. However, once the the generic simplicity holds for $S$, it follows from Theorem \ref{thm:Cayley!generic} that it also holds for every $H \subseteq G\setminus\{e\}$ containing $S$.
\end{remark}

To provide a deeper understanding of the theorem, we apply the established criterion in the following examples. These examples showcase situations where the irreducibility property holds true, as well as cases where it does not. By exploring both scenarios, we can gain a more nuanced understanding of the theorem and its range of applicability.

\begin{example} \label{ex:abelian}
    Let $G$ be a \textbf{finite abelian group}. It is known that it is equivalent to the case where every $(W,\rho) \in \operatorname{Irr}(G,\C)$ has dimension $1$ (see \cite{serre1977}, \S 3.1). Denote by $(W_j,\rho_j)$, $j \in \{0,1, \dots, n-1\}$, all distinct irreducible representations of $G$. Then, one  can see that, for all $\alpha \in \mathbb{R}^{S'}$,
    \begin{equation} \label{DWj:func}
        D_{W_j}(\alpha):= -\sum_{s \in S}\alpha_s(\rho_j(s)- \rho_j(e))
    \end{equation}
    is a scalar. Set $\kappa_j$ to be the constant $\kappa_j:= D_{W_j}(\alpha)$. Therefore $P_{W_j}(\alpha)(x)=x- \kappa_j$ for all $j \in \{0,1, \dots, n-1\}$. Since the eigenvalues are simple, items $(ii)$ and $(iii)$ of Theorem \ref{thm:Cayley!generic} are satisfied.

    Thus, the generic property for the eigenspaces of $\Delta_w$ is conditioned to the existence of $\alpha \in \mathbb{R}^{S'}$ satisfying
    \begin{equation} \label{kappa:neq}
        \kappa_j \neq \kappa_k
    \end{equation}
    for all distinct $j,k \in \{0,1, \dots, n-1\}$ with $W_j^\ast \not \cong W_k$.
\end{example}

\begin{example} \label{ex:cyclic.dir}
    Let $\mathbf{C}_n= \langle r \rangle$ denote the \textbf{cyclic group} of order $n \geq 2$. We know from \cite[\S 5.1]{serre1977} that the $n$ irreducible linear representations are given by
    \[
        \rho_j(r^k) = e^{2 \pi\frac{jk}{n}i}.
    \]

    Let us now fix a symmetric generating set $S \subseteq \mathbf{C}_n\setminus\{e\}$ such that $r \in S$  with $\langle r\rangle= \mathbf{C}_n$. Set $\alpha \in \mathbb{R}^{S'}$ with $\alpha_r=\alpha_{r^{-1}}= 1$ and $\alpha_s = 0$ for $s \in S \setminus \{r^{\pm1}\}$. It follows that, for $n \geq 3$,
    \[
        \kappa_j= -\sum_{r^k \in S} \alpha_{r^k}(\omega^{jk} - 1)= 2 - 2\cos\left(\frac{j}{n}\pi\right) = \kappa_{n-j},
    \]
    while $\kappa_j= \big(1-(-1)^j\big)$ for $n=2$. The group $\mathbf{C}_n$ is abelian, Thus, since $(W_j^\ast,\rho_j^\ast) \cong (W_{n-j},\rho_{n-j})$ for all $j \in \{1, \dots, n-1\}$, the desired conclusion follows from \eqref{kappa:neq} and $\Delta_w$ is real $\mathbf{C}_n$-simple on $\mathcal{C}(\mathbf{C}_n, S)$.

    We will consider now a case where $\mathbf{C}_n$ the element $r$ is not included in the generating set. Let $n=m_1m_2$ with $m_1$ and $m_2$ two coprime numbers, then $\mathbf{C}_n =\langle r_1,r_2 \rangle \simeq \mathbf{C}_{m_1}\times\mathbf{C}_{m_2}$ where $r_k := r^{n/m_k}$. Set $H:=\{r_1^{\pm1},r_2^{\pm1}\} \subseteq L$ and $\alpha_s = 0$ for $s \in L \setminus H$. Let now
    \[
        \kappa_j^{(k)}:= \left\{\begin{array}{cl}
           2 \alpha_{r_k} \big(1-\cos(2\pi j /m_k)\big),  & \text{if } m_k\neq 2, \\
            \big(1-(-1)^{j}\big)\alpha_{r_k} & \text{if } m_k= 2.
        \end{array}\right.
    \]
    Then, $\kappa_j= \kappa_j^{(1)} + \kappa_j^{(2)}$ for $j \in \{0, 1, \dots, n-1\}$. Observe now that $\kappa_j^{(k)} = \kappa_\ell^{(k)}$ when $j = \pm \ell \mod m_k$. By the Chinese remainder Theorem, for  all $\ell_1 \in \{0, 1, \dots, m_1-1\}$ and $\ell_2 \in \{0, 1, \dots, m_2-1\}$, there exists \[j' = j''= \ell_1 \mod m_1 \quad \text{and}\quad j' = -j''= \ell_2 \mod m_2 \quad\]
    with $j',j'' \in \{0, 1, \dots, n-1\}$ unique and, since $m_1$ and $m_2$ are coprime, one has that $j' = \pm j'' \mod n$ exactly when $\ell_1=0$ or $\ell_2 =0$. It implies that one can find $\alpha \in \mathbb{R}^{L'}$ such that \eqref{kappa:neq} is satisfied. Therefore, $\mathcal{W}_S$ is residual in $\mathcal{L}_S$ for $\Delta_w$ on $\mathcal{C}(\mathbf{C}_n,L)$.
\end{example}

\begin{example} \label{ex:Klein-four}
    Let $\mathbf{K}\simeq \mathbf{C}_2\times\mathbf{C}_2$ denote the \textbf{Klein four-group}. Recall that $\mathbf{K} = \langle a,b \rangle$ consists of $4$ elements such that $a^2 = b^2 = (ab)^2 = e$. We can verify that the four linear irreducible representations of $\mathbf{K}$ are
    \begin{table}[H]
    \centering
    \begin{tabular}{l|rrrr}
                    & $e$ & $a$ & $b$ & $ab$ \\ \hline \hline
        $\rho_e$ & $1$   & $1$   & $1$   & $1$    \\
        $\rho_a$ & $1$   & $1$   & $-1$  & $-1$   \\
        $\rho_b$ & $1$   & $-1$  & $1$   & $-1$   \\
        $\rho_{ab}$ & $1$   & $-1$  & $-1$   & $1$
    \end{tabular}.
    \end{table}
    
    We consider $S$ to be a generator of $G$. Fix, without loss of generality, $r,t \in S$ such that $r^{\pm1}\neq t^{\pm1}$. We may choose $\alpha_r =1$. Set $\alpha_t=2$ and $\alpha_s=0$ for the possible remaining $s \in S \setminus \{r,t\}$. Therefore, the condition \eqref{kappa:neq} is verified and the generic property holds for every generator set $S$.
    
    \begin{figure}[H]
    \centering
    \includegraphics[scale=0.3]{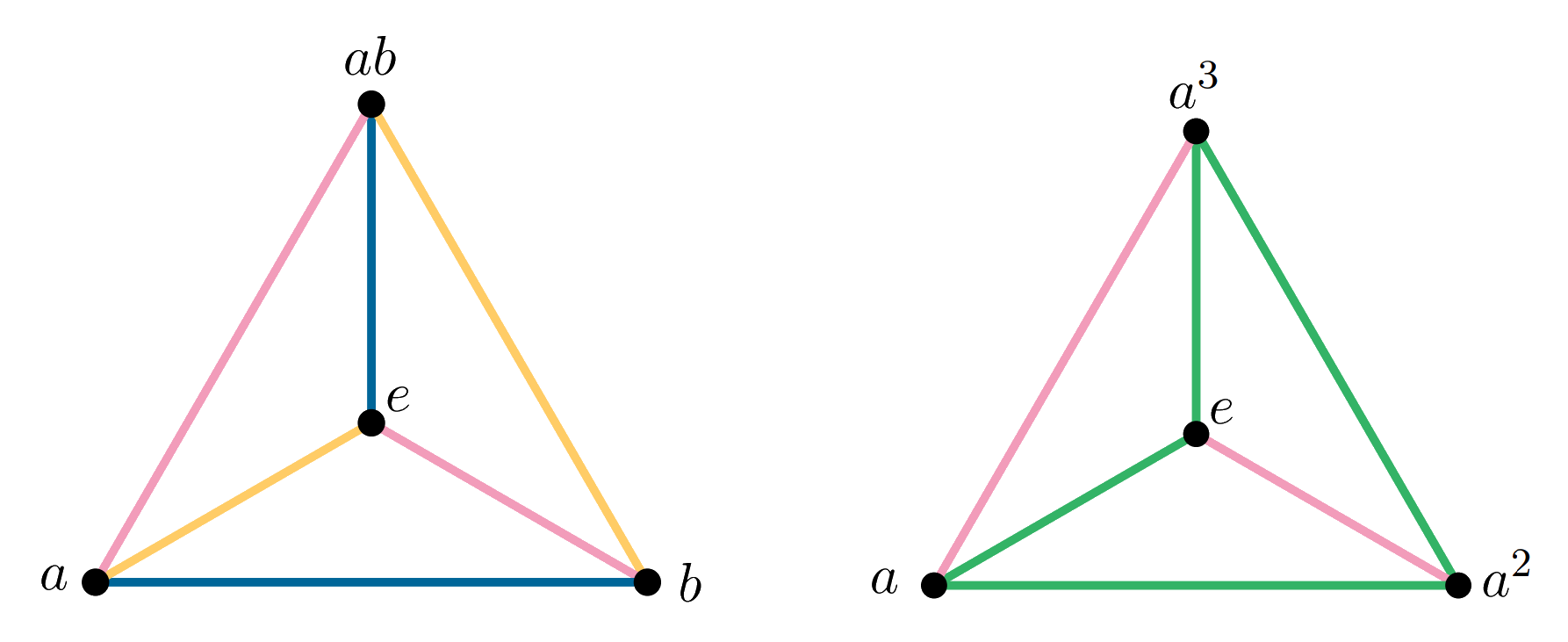}
        \caption{Tetrahedrons associated with the Cayley graphs of the Klein four-group generated by the set $S=\{a,b,ab\}$ (left) and with the cyclic group $\mathcal{C}(\mathbf{C_4}, \{r,r^2,r^3\})$ (right) (see Examples \ref{ex:cyclic.dir} and \ref{ex:Klein-four}).}
        \label{fig:tetrahedron}
    \end{figure}
\end{example}

The examples above shows us that the generic property holds for cyclic groups and $\mathbf{C}_2 \times \mathbf{C}_2$. We will exhibit an example below of an abelian group for which the real $G$-simplicity fails.  

\begin{example}
    Consider $G = \mathbf{C}_3\times \mathbf{C}_3 = \langle (r,e),(e,t) \rangle$ where $r^3=e=t^3$ and $r \neq e \neq t$. Let $\omega= \exp(2\pi i/3)$. The nine irreducible representations are given by
    \[
        \rho_{j_1,j_2}(r^{k_1},t^{k_2})= \omega^{j_1 k_1+ j_2 k_2}
    \]
    with $j_1,j_2,k_1,k_2 \in \{0,1,2\}$. Set $S= \{(r^{\pm1},e), (e, t^{\pm1})\}$. Hence,
    \[D_{W_{j_1,j_2}}(\alpha) = 2\alpha_r\big(1 - \cos(2 \pi j_1 /3)\big) + 2\alpha_t\big(1 - \cos(2 \pi j_2 /3)\big).\]

    It then follows that
    \[
        \kappa_{1,1} = \kappa_{1,2} = 3 \alpha_r + 3\alpha_t = \kappa_{2,1} = \kappa_{2,2},
    \]
    while $W_{1,1}^\ast \ncong W_{1,2}$, $W_{1,1}^\ast \simeq W_{2,2}$, and $W_{1,2}^\ast \cong W_{2,1}$. Since $G$ is abelian and \eqref{kappa:neq} does not hold, $\Delta_w$ is not real $G$-simple on $\mathcal{C}(G,S)$.
\end{example}

\begin{remark}
The example above shows us a case where we do not have the generic property. However, when $G$ is abelian we can find a weaker version of our result in such a way that we will reacquire the generic property for the Laplacian. 

In fact, the condition $w_s = w_{s^{-1}}$ is too strong when considering all functions in $\A$. We must then consider a subspace $\A_S$ of $\A$ that is more compatible with this property. More specifically, we want the identity
\begin{equation} \label{tag I}
    w_s (f(xs)-f(x))= w_{s^{-1}}(f(xs^{-1})-f(x)).
\end{equation}

If \eqref{tag I} is true, then the terms associated with $w_s$ and $w_{s^{-1}}$ will produce the same effect in the Laplacian.

Define, for each representation $(W, \rho)$, the subrepresentation

\[
\A_S^{\rho} : = \bigcap\limits_{s \in S} \text{ker} \big( \rho(s) - \rho(s^{-1}) \big).
\]
We note that condition \eqref{tag I} holds exactly for the functions $f \in \A_S^{R}$ with respect to the right regular representation on $\A$. Thus, we can define $\A_S := \A_S^R$.
 
Now we can replace $\Delta_{w'}$ by $\Delta_{w'}|_{\A_{S}^R}$ in Theorem \ref{thm:Cayley!generic}. Therefore, the subset $\mathcal{W}_S$ of $\mathcal{L}_S$ that turns $\Delta_{w'}|_{\A_{S}^R}$  into a real $G$-simple operator on $\mathcal{C}(G,S)$ is residual. Since the generic property holds on $\mathcal{A}_S$, then $\mathcal{A}_S$ has the Peter-Weyl decomposition\[\mathcal{A}_S \cong \bigoplus_{W \in \operatorname{Irr}(G,\mathbb{C})}W\]which coincides with the eigendecomposition of $\Delta_w\mid_{\mathcal{A}_S}: \mathcal{A}_S \to \mathcal{A}_S$ for a generic weight $w$.
\end{remark}

\begin{example} \label{ex:dihedral}
    Let $\mathbf{D}_n$ the \textbf{dihedral group} of order $2n$ with $n \geq 3$. We know from \cite[\S 5.3]{serre1977} that there exists at most four irreducible representations of degree $1$ (two in the case that $n$ is odd and four if $n$ is even) and the others are of degree $2$.
    
    We will consider $S$ containing a reflection $t$ and a rotation $r$ such that $\mathbf{C}_n = \langle r \rangle $ and
    \[
    \mathbf{D}_n = \mathbf{C}_n \cup t\mathbf{C}_n
    \]
    where $t\mathbf{C}_n$ denotes the coset $\{tr^k; \; k=0,\dots,n-1 \}$. Note that $\{r^{\pm 1},t\} \subseteq S$ is itself a generating set $\mathbf{D}_n$. The irreducible representations of degree $1$ satisfy
    \[
    \rho_{j,k}(r) = j  \quad \text{and} \quad \rho_{j,k}(t) = k 
    \]
    for $(j,k) \in \{ (1,1),(1,-1) \}$ if $n$ is odd and $(j,k) \in \{ (1,1),(1,-1),(-1,1),(-1,-1) \}$ if $n$ is even. We will denote these representations by $W_{j,k}$.
    
    Consider $m$ to be an integer and set and $\omega := e^{2\pi i /n}$. Then the irreducible representations of degree $2$ are $W_m = (\mathbb{C}^2, \rho_m)$ for $0<m<n/2$ with
    \[
    \rho_m(r^k) = 
    \begin{pmatrix*}[c]
         \omega^{mk}& 0  \\
         0& \omega^{-mk} 
    \end{pmatrix*}, 
    \quad \text{and} \quad
    \rho_m(tr^k) =
    \begin{pmatrix*}[c]
         0 & \omega^{-mk}  \\
         \omega^{mk}& 0 
    \end{pmatrix*}.
    \] 
     
     Let us fix $\alpha \in \mathbb{R}^{S'}$ with $\alpha_r= \alpha_{r^{-1}}=1/2$, $\alpha_t = 3$, and $\alpha_s = 0$ for $s\in S \setminus \{ r^{\pm 1},t\}$. Thus
    \[
        D_{W_{j,k}}(\alpha) = 4- \frac{1}{2}(j+j^{n-1})-3k.
    \]
    Therefore,
    \begin{itemize}
        \item $(j,k)= (1,1)$ implies that $D_{W_{j,k}}(\alpha)=0$.
        \item $(j,k)= (1,-1)$ implies that $D_{W_{j,k}}(\alpha)=6$.
        \item $n$ even and $(j,k)= (-1,1)$ implies that $D_{W_{j,k}}(\alpha)=2$.
        \item $n$ even and $(j,k)= (-1,-1)$ implies that $D_{W_{j,k}}(\alpha)=8$.
    \end{itemize}
    Observe that
     \[
     D_{W_m}(\alpha) = 
     \begin{pmatrix*}[c]
         4 -\cos(2 \pi m/n)& -3  \\
         -3 & 4 -\cos(2 \pi m/n) 
    \end{pmatrix*}
     \]
    has distinct eigenvalues $\lambda_1^{(m)}= \cos(2\pi m/n)+1$ and $\lambda_2^{(m)}= \cos(2\pi m/n)+4$. Since the cosine function is strictly decreasing in $(0,\pi)$, $\lambda_1^{(m)} \neq \lambda_1^{(m')}$ and $\lambda_2^{(m)} \neq \lambda_2^{(m')}$ when $m \neq m'$.

    Suppose now, by contradiction, that there exists $m \in \mathbb{Z}$ with $0<m<\pi/2$ such that
    \[
        \lambda_1^{(m)}=\lambda_2^{(m')}.
    \]
    Then $\cos(2\pi m/n)- \cos(2\pi m'/n) = 3$, which is impossible. Hence, $\lambda_1^{(m)}\neq\lambda_2^{(m')}$ for all $m, m' \in \mathbb{Z}$ with $0< m, m'<\pi/2$.
    
    Let $\lambda_1^{(m)} \in \mathbb{Z}$ or $\lambda_2^{(m)} \in \mathbb{Z}$. Then $m=n/4\in \mathbb{Z}$ and hence $\lambda_1^{(m)}=1$ and $\lambda_2^{(m)}=4$.
    
    Therefore, for all $W, W' \in \operatorname{Irr}(\mathbf{D}_n,\complexset)$, one has $a_{W, W'} \not\equiv 0$ satisfying condition $(i)$. Conditions $(ii)$ and $(iii)$ of Theorem \ref{thm:Cayley!generic} are immediately satisfied since all $D_W(\alpha)$ have simple eigenvalues. 
\end{example}

\begin{example}
    The \textbf{alternating group} $\mathbf{A}_4$ can be defined by $\mathbf{A}_4= \langle t,x\rangle$ with $t=(123)$ and $x= (12)(34)$. Let $t,x \in S$. There exists four irreducible representations (see \cite[\S 5.7]{serre1977} for details). Let $\omega:= e^{2\pi i/3}$. The irreducible representations of degree $1$ are
    \[
        \rho_j (t)=\omega^j \quad \text{and} \quad \rho_j (x)=1, \quad \text{for }j \in \{0,1,2\}.
    \]
    The remaining irreducible representation is of degree $3$. We will denote it by $\rho_3$ and it can be generated by
    \[
        \rho_3(t)=
        \begin{pmatrix*}[r]
            1 & 0 & 0\\
            0 & 0 & 1 \\
            -1 & -1 & -1 
        \end{pmatrix*}
        \quad \text{and} \quad
        \rho_3(x)=
        \begin{pmatrix*}[r]
            0 & 1 & 0\\
            1 & 0 & 0 \\
            -1 & -1 & -1.
        \end{pmatrix*}
    \]
   
    Let $\alpha_t=\alpha_x:=1$ and $\alpha_s:=0$ for $s \in S\setminus\{t,x\}$. Thence,
    \begin{gather*}
        P_{W_0}(\alpha)(\lambda)=\lambda, \quad P_{W_1}(\alpha)(\lambda) = P_{W_2}(\alpha)(\lambda) =\lambda - 2\big(1- \cos({2}\pi/3)\big), \\
        and \quad  P_{W_3}(\alpha)(\lambda)=- (\lambda-1)(\lambda-4)(\lambda-5).
    \end{gather*}
    
    Since $W_1^\ast\cong W_2$, one can easily verify items $(i)$ $(ii)$ and $(iii)$ of Theorem \ref{thm:Cayley!generic}. It thus follows that the generic irreducibility of the eigenspaces of $\Delta_w$ holds on $\mathcal{C}(\mathbf{A}_4,S)$.
\end{example}

\begin{example}
 Let $\mathbf{S}_4$ be the \textbf{permutation group of four elements} $\{1,2,3,4\}$. The order of $\mathbf{S}_4$ is $4!=24$. Consider $S=\{\tau,\sigma^{\pm1}\}$ with $\tau = (34)$ and $\sigma = (123)$. Then $S$ generates $\mathbf{S}_4$.  In fact, $\mathcal{C}(\mathbf{S}_4,S)$ is the truncated cube (see Fig. \ref{fig:s4}).  We know from \cite[\S 5.8]{serre1977} that there are two irreducible representations of degree $1$, one of degree $2$, and other two of degree $3$.
 
We present below the five nonequivalent irreducible representations evaluated in $\tau$ and $\sigma$ (see \cite{sagan2001,serre1977,zhao2008} for more details). 
 
\paragraph{Representations of degree 1}
 \begin{itemize}
     \item $(W_{11}, \rho_{11})$ with $\rho_{11}$ the trivial irreducible representation, \textit{i.e.}, $\rho_{11} \equiv 1$.

     \item $(W_{12}, \rho_{12})$ with $\rho_{12}$ the sign representation given by $\rho_{12}(\tau)=-1$ and  $\rho_{12}(\sigma)=1$.
 \end{itemize} 
 
\paragraph{Representation of degree 2}
\begin{itemize}
    \item $(W_{21}, \rho_{21})$ such that $\rho_{21}(\tau) = \begin{pmatrix*}[r]
    1 & 0  \\
     -1&-1 
\end{pmatrix*}$ and
$\rho_{21}(\sigma) = \begin{pmatrix*}[r]
    0 & 1  \\
     -1&-1 
\end{pmatrix*}$.
\end{itemize}
 
\paragraph{Representations of degree 3}
\begin{itemize}
    \item $(W_{31}, \rho_{31})$ satisfies $\rho_{31}(\tau) =
 \begin{pmatrix*}[r]
    1 & 0 & 0  \\
     0 & 0 & 1 \\
     0 & 1 & 0
\end{pmatrix*}$ and 
$\rho_{31}(\sigma) = \begin{pmatrix*}[r]
    -1 & -1 & -1  \\
     1 & 0 & 0 \\
     0 & 0 & 1
\end{pmatrix*}$.

\item $(W_{32}, \rho_{32})$ satisfies $\rho_{32}(\tau) = \begin{pmatrix*}[r]
    -1 & 0 & 0  \\
     0 & 0 & 1 \\
     0 & 1 & 0
\end{pmatrix*}$ and
$\rho_{32}(\sigma) = \begin{pmatrix*}[r]
    0 & 1 & 0  \\
     -1 & -1 & 0 \\
     0 & 1 & 1
\end{pmatrix*}$.
\end{itemize}

Let $\alpha \in \mathbb{R}^{S'}$ be associated with the generating set $S= \{\tau,\sigma^{\pm1}\}$. Note that
\[D_{W_{21}}(\alpha) = \begin{pmatrix*}[c]
    3\alpha_\sigma & 0   \\
    \alpha_\tau & 3\alpha_\sigma +2\alpha_\tau 
\end{pmatrix*} \quad \text{and} \quad D_{W_{31}}(\alpha) = \begin{pmatrix*}[c]
    3\alpha_\sigma & 0 & \alpha_\sigma  \\
    0 & 3\alpha_\sigma +\alpha_\tau  
    & \alpha_\sigma -\alpha_\tau
 \\ 0 & -\alpha_\tau & -\alpha_\sigma + \alpha_\tau
\end{pmatrix*}.\]

Then $3\alpha_\sigma$ is a common eigenvalue for $D_{W_{31}}(\alpha)$ and $D_{W_{32}}(\alpha)$ and arbitrary $\alpha$. It follows that $a_{W_{21},W_{31}}=0$ and, therefore, $\Delta_w$ is not $G$-simple on $\mathcal{C}(\mathbf{S}_4,S)$.

However, consider now $H \subseteq \mathbf{S}_4 \setminus \{e\}$ symmetric such that $\{\tau, \sigma^{\pm1}, \eta^{\pm1}\} \subseteq H$ with $\eta := \sigma \tau = (1234)$. Let $\beta \in \mathbb{R}^{T'}$ for the generating set $T$ be such that $\beta_t = 1$ when $t \in \{\tau, \sigma^{\pm1},\eta^{\pm1}\}$ and $\beta_s = 0$ when $s \not\in \{\tau, \sigma^{\pm1},\eta^{\pm1}\}$. One can now verify that all eigenvalues of $D_{W}(\beta)$ for $W \in \operatorname{Irr}(\mathbf{S}_4,\complexset)$ are distinct with multiplicity one, namely:
\begin{align*}
 \lambda_{11} =0, && \lambda_{12} = 4,\phantom{- \sqrt{3},}   &&  \lambda_{21} = 6+ \sqrt{3},  && \lambda_{21}' = 6- \sqrt{3},  \\
    \lambda_{32} = 6, && \lambda_{32}' = 4+ \sqrt{7},   &&  \lambda_{32}'' = 4- \sqrt{7},
\end{align*}
\begin{align*}
    \lambda_{31} &= \frac{1}{3} \left(\sqrt[3]{-8-3i\sqrt{237}} + \sqrt[3]{-8+3i\sqrt{237}} +16\right),\\
    \lambda_{31}' &= \frac{1}{6} \left(-(1+i\sqrt{3})\sqrt[3]{-8-3i\sqrt{237}} + (-1+i\sqrt{3})\sqrt[3]{-8+3i\sqrt{237}} +32\right), \text{ and}\\
    \lambda_{31}'' &= \frac{1}{6} \left((1-i\sqrt{3})\sqrt[3]{-8-3i\sqrt{237}} -(1+i\sqrt{3})\sqrt[3]{-8+3i\sqrt{237}} +32\right).
 \end{align*}

 \begin{figure}[htb!]
     \centering
     \includegraphics[scale=0.234]{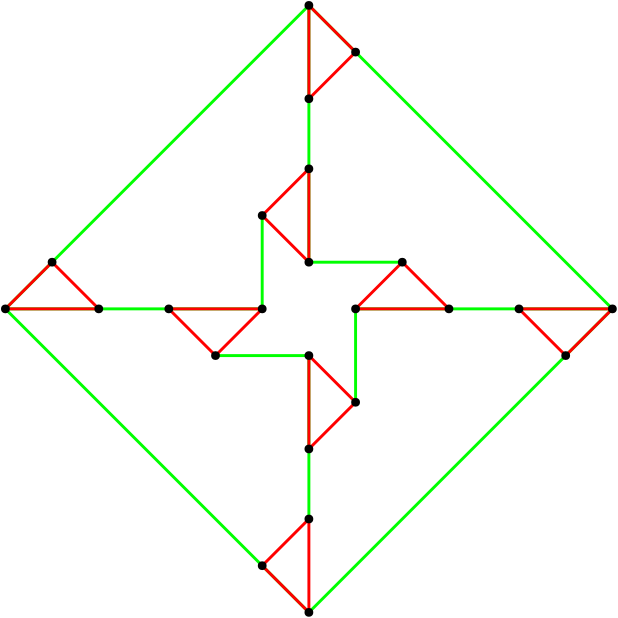} \hspace{0.4cm}
     \includegraphics[scale=0.155]{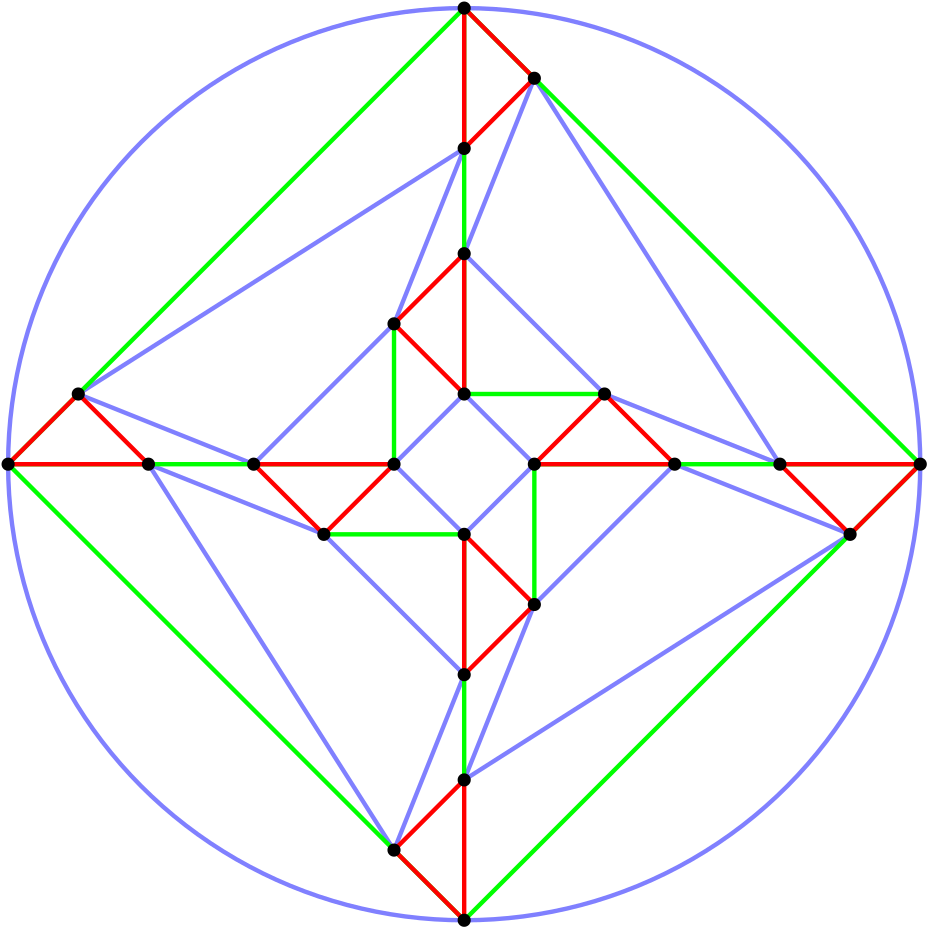}
     \caption{The truncated cube $\mathcal{C}(\mathbf{S}_4,S)$ (left) and $\mathcal{C}(\mathbf{S}_4,T)$ with $T= \{\tau,\sigma^{\pm1},\eta^{\pm1}\}$ (right), the graph union of the truncated cube with the rhombicuboctahedron.}
     \label{fig:s4}
 \end{figure}

Therefore, the generic irreducibility of the eigenspaces of $\Delta_w$ holds on $\mathcal{C}(\mathbf{S}_4,H)$.
\end{example}

\section{\texorpdfstring{The operator $\Lw$}{The operator \texttt{L}{\it w}}} \label{sec.Lw}

We introduce a new operator motivated by the particular differential structure of Cayley graphs. Let us define the linear operator $\Lw$ for $f \in \mathcal{A}$ on $\mathcal{C}(G,S)$ as
\[
\Lw f(x) =  \sum\limits_{s\in S} w_s\big(f(xs^2)-2f(xs)+f(x)\big).
\]

The standard differential calculus on graphs yields the weighted Laplacian $\Delta_w$ given by a metric tensor determined by $w$. The operators are constructed based on the notion of discrete derivative, connections, divergence and other concepts from Riemannian geometry.

Let $\mathcal{G}=(V,E)$ a graph and $f \in \mathcal{A}$ with $f = \sum_{x \in G}f(x)\delta_x$. The concept of direction is essential to define derivatives. Each edge $\{x,y\} \in E$ can be seen as two directed edges $(x,y)$ and $(y,x)$. The derivative of $f$ on $(x,y)$ is given by the discrete derivative $\partial_{x,y}f:= \big(f(y)-f(x)\big)\delta_x$.

However, Cayley graphs are rather special, they are regular and vertex-transient. In particular, each vertex of $\mathcal{C}(G,S)$ has its neighbours determined by $S$. In other words, each $s \in S$ determines a direction for all $x \in G$, so we can now define the derivative of $f$ in the direction $s\in S$ by $\partial_s f := \sum_{x\in G}\big(f(xs)-f(x)\big)\delta_x$. Note that
\[\partial_{x,xs}f(x) = \partial_sf(x).\]
On the other hand,
\[\partial_{x,y}\big(\partial_{x,y}f\big) = \partial_{x,y} \big( (f(y)-f(x)) \delta_x\big) = -\big(f(y)-f(x)\big) \delta_x\]
while
\[\partial_{s}\big(\partial_{s}f\big) = \partial_{s} \left( \sum_{x \in G}\big(f(xs)-f(x)\big) \delta_x\right) = \sum_{x\in G}\big(f(xs^2)-2f(xs)+f(x)\big) \delta_x.\]

This particularity motivates us to define $\Lw$ as above. The operator can be obtained as the divergence of the gradient by applying standard differential calculus techniques, but it is beyond the scope of this text. We will prove now that $\Lw$ is symmetric with respect to $\langle\cdot,\cdot\rangle$.

\begin{lemma}
    Let $G$ be a finite group with $S$ a symmetric generating set. Then, for all $w \in \mathcal{L}_S$ and all $f,g \in \mathcal{A}$,
    \[\langle \Lw f, g\rangle = \langle f,\Lw g\rangle.\]
\end{lemma}
\begin{proof}
    We proceed similarly to the proof of Lemma \ref{lm:Laplacian.symmetric}. Since $G$ acts transitively on itself, $G=Gs^m$ for all $m \in \mathbb{Z}$, then \[\sum_{x \in G} \sum_{s\in S} f(xs^m)\overline{g(x)} = \sum_{x \in G} \sum_{s\in S} f(x)\overline{g(xs^{-m})}.\]

    Recall once again that $s^{\pm1} \in S$ and $w_s=w_{s^{-1}}$. Therefore,
    \begin{align*}
        \langle\Lw f,g\rangle &= \sum_{x \in G} \sum_{s \in S} w_s\big( f(xs^2)-2f(xs) + f(x) \big)\overline{g(x)} \\
        &= \sum_{x \in G} f(x)\left(\overline{\sum_{s \in S} w_{s}\left(g(xs^2) -2g(xs)+ g(x)\right)}\right) = \langle f, \Lw g\rangle. 
    \end{align*}
\end{proof}

Let us write for each $(W,\rho) \in \operatorname{Irr}(G,\C)$ the operator $\mathtt{D}_W: \mathbb{R}^{S'} \to \operatorname{End}(W)$ given by
\[
\mathtt{D}_W(\alpha) := \sum\limits_{s \in S}\alpha_{s} \big(\rho(s^2)-2\rho(s)+\rho(e) \big) = \sum\limits_{s \in S}\alpha_{s} \big(\rho(s)-\rho(e) \big)^2.
\]
Observe now that Corollary \ref{corolario 3.1 - Schueth} holds for the operator $\Lw$ and $\mathtt{L}_w^W := \mathtt{D}_W(w)$ replacing $\Delta_w$ and $\Delta_w^W$, respectively. By rewriting $D_W$ as $\mathtt{D}_W$ from \eqref{DW operator} onwards, the following theorem is straightforward:

\begin{theorem} \label{thm:Cayley!Lw}
Theorems \ref{thm:Cayley!generic} is also true for the operator $\Lw$.
\end{theorem}

\begin{remark}
 In fact, it is now easy to verify from the precedent constructions that, with some elementary adaptations, Theorem \ref{thm:Cayley!Lw} also holds for a more arbitrary class of operators $R(z)$ satisfying $z \in \mathbb{C}[G]$ self-adjoint.  
\end{remark}

Note that Remark \ref{rem:Ssubset:irr} is also true for $\Lw$ as consequence of Theorem \ref{thm:Cayley!Lw}. The examples from Section \ref{main results} can be easily addapted for $\Lw$. For instance, if $G$ is a finite abelian group it suffices to write $\mathtt{D}_W(\alpha)= \kappa_j$ in Example \ref{ex:abelian} with 
\begin{equation*} \label{kappa:abelian:Lw}
    \kappa_j= \sum_{s \in S} \alpha_s\big(1-\rho_j(s)\big)^2 \; ,
\end{equation*}
replacing \eqref{DWj:func} and the conclusion \eqref{kappa:neq} for $\Lw$.

A special case for $\Lw$ is when $S \subseteq G\setminus \{e\}$ is such that $s^2=e$ for every $s \in S$. Then $G \simeq \mathbf{C}_2\times \mathbf{C}_2 \times \dots \times \mathbf{C}_2$ (\emph{e.g.}, $\mathbf{K}$ the Klein four group) and $|G|=2^n$. Therefore $\Lw= 2 \Delta_w$ and the generic spectrum is completely conditioned to Theorem \ref{thm:Cayley!generic}.

\section*{Acknowledgements}

We sincerely thank the anonymous referees for their invaluable feedback and constructive criticism, which significantly improved this manuscript. Their meticulous review and insightful suggestions have been instrumental in improving our work.

\bibliographystyle{siam}
\bibliography{references}

\end{document}